\documentclass[11pt,leqno]{amsart}
\usepackage{amsmath}
\usepackage{amssymb}

\def\underset#1#2{{\mathrel{\mathop {{}_{} {#2}}\limits_{{#1}_{}}}}}
\def\upplim_#1{\underset{#1}{\overline\lim}\;}
\def\lowlim_#1{\underset{#1}{\underline\lim}\;}
\newtheorem{definition}[equation]{Definition}
\newtheorem{claim}[equation]{\indent{\it Claim}\rm }

\newtheorem{lemma}[equation]{Lemma}

\newtheorem{remark}[equation]{\indent\rm {\it Remark}}
\newtheorem{theorem}[equation]{Theorem}

\newcommand{\C}{{\mathbb{C}}}

\renewcommand{\c}{{\bf c}}
\newcommand{\x}{{\bf x}}
\newcommand{\y}{{\bf y}}
\renewcommand{\P}{{\mathbb{P}}}
\newcommand{\Q}{{\mathbb{Q}}}
\newcommand{\R}{{\mathbb{R}}}

\numberwithin{equation}{section}

\title[Some generalizations of Schmidt's subspace theorem]{Some generalizations of Schmidt's subspace theorem} 

\author{Si Duc Quang}

\begin{document}

\maketitle 

\begin{abstract}
The aim of this paper is twofold. The first is to give a quantitative version of Schmidt's subspace theorem for arbitrary families of higher degree polynomials. The second is to give a generalization of the subspace theorem for arbitrary families of closed subschemes in algebraic projective varieties. 
\end{abstract}

\def\thefootnote{\empty}

\section{Introduction}
The theory of Schmidt's subspace theorem is an important branch of Diophantine approximation with many significant applications to Diophantine equations, Diophantine geometry. The first versions of the subspace theorem were established by Schmidt in the 70's of the last century (see \cite{Sch72,Sch75,Sch89}). After that, his results have been improved and generalized by many authors. In 2008, Evertse and Ferretti gave a very strong generalization of these results. In order to state their result, we recall the following standard notions in Diophantine approximation.

Let $K$ be a number field. Denote by $M_K$ the set of places (equivalence classes of absolute values) of $K$ and by $M^\infty_K$ the set of archimedean places of $K$. For each $v\in M_K$, we choose the normalized absolute value $|.|_v$ such that $|.|_v=|.|$ on $\mathbb Q$ (the standard absolute value) if $v$ is archimedean, and $|p|_v=p^{-1}$ if $v$ is non-archimedean and lies above the rational prime $p$. Denote by $K_v$ the completion of $K$ with respect to $v$ and set
 $$\|x\|_v:=|x|^{n_v}_v\ \text{ where }n_v := [K_v :\mathbb Q_v]/[K :\mathbb Q].$$
The absolute logarithmic height of a projective point ${\bf x}=(x_0:\cdots :x_N)\in\P^N(K)$ is defined by
$$h({\bf x}):=\sum_{v\in M_K}\log\max\{\|x_0\|_v,\ldots,\|x_N\|_v\}.$$
Also, we choose an extension of $|.|_v$ to ${\overline{\Q}}$ which amounts to extending $|.|_v$ to the algebraic closure $\overline{K}_v$ of $K_v$ and choosing an embedding ${\overline{\Q}}$ into $\overline{K}_v$.  For ${\bf x}=(x_0,\ldots,x_N)\in{\overline{\Q}}^{N+1}$ we put $\|{\bf x}\|_v:=\max\{|x_0|_v,\ldots,|x_N|_v\}$. 

For a given system $f_0,\ldots,f_m$ of homogeneous polynomials in ${\overline{\Q}}[x_0,\ldots,x_N]$, define 
$$h(f_0,\ldots,f_m):=h({\bf a}),$$
where ${\bf a}$ is a vector consisting of the nonzero coefficients of $f_0,\ldots,f_m$. 
The height of a projective subvariety $X$ of $\P^N$ defined over ${\overline{\Q}}$ is defined by 
$$h(X):=h(F_X),$$
where $F_X$ is the Chow form of $X$ (see Section $\S $2).

Denote by $G_K$ the Galois group of ${\overline{\Q}}$ over $K$. For each ${\bf x}=(x_0,\ldots ,x_N)\in {\overline{\Q}}^{N+1},\sigma\in G_K$ we write $\sigma ({\bf x})=(\sigma (x_0),\ldots ,\sigma(x_N))$.  Evertse and Ferretti proved the following quantitative version of the subspace theorem.

\vskip0.2cm
\noindent
{\bf Theorem A} (Evertse and Ferretti \cite{EF08}, Theorem 1.3). {\it Let $\delta$ be a real with $0<\delta<1$, $K$ a number field, $S$ a finite set of places of $K$ of cardinality $s$, $X$ be a projective subvariety of $\P^N$ defined over $K$ of dimension $n\ge 1$ and of degree $d$, and $f^{(v)}_0,\ldots,f^{(v)}_n\ (v\in S)$ be a system of homogeneous polynomials in ${\overline{\Q}}[x_0,\ldots, x_N]$. 
Assume that
$$X({\overline{\Q}})\cap\{f^{(v)}_0=0,\ldots,f^{(v)}_n=0\}=\emptyset\text{ for all }  v\in S.$$
Then there are homogeneous polynomials $G_1,\ldots,G_u\in K[x_0,\ldots,x_N]$ with
$$u\le A_1, \deg G_i\le A_2\text{ for }i=1,\ldots,u$$
which do not vanish identically on $X$, such that the set of $\x\in X({\overline{\Q}})$ with
\begin{align*}
&\log\left(\prod_{v\in S}\prod_{i=0}^n\max\limits_{\sigma\in G_K}\frac{|f^{(v)}_i(\sigma(\x))|_v^{1/\deg f^{(v)}_i}}{\|\sigma(\x)\|_v}\right)\le -(n+1+\delta)h(\x),\\
&h(\x)\ge A_3H 
\end{align*}
is contained in $\bigcup_{i=1}^u(X\cap\{G_i=0\})$, where $A_1,A_2,A_3,H$ are explicitly estimated.}

As we known that, there is a close relation between the Diophantine approximation and Nevanlinna theory via the dictionary of Vojta, in which Schmidt subspace theorem correspondences to the second main theorem. In order to get the a general second main theorem for arbitrary families of hypersurfaces, in \cite{Qpcf}, the author introduces the notion of distributive constant of a family of hypersurfaces. We state its reformulation as follows.

\vskip0.2cm
\noindent
\textbf{Definition B} (see \cite[Definition 3.3]{Qpcf}, reformulation) {\it Let $X$ be a subvariety of a projective variety $V$ and $\mathcal D=\{D_{1}, \ldots, D_{q}\}$ a family of $q$ divisors in $V$ such that $X\not\subset D_j\ (1\le j\le q)$. The distributive constant of the family $\mathcal D$ with respect to $X$ is defined by
$$ \delta_{\mathcal D,X}=\max_{\emptyset\ne J\subset\{1,\ldots,q\}}\frac{\sharp J}{\dim X-\dim X\cap\bigcap_{j\in J}D_j},$$
where $\dim\emptyset=-\infty$.}

We now generalize Theorem A to the following theorem.
\begin{theorem}\label{1.1}
Let $\delta$ be a real with $0<\delta<1$, $K$ a number field, $S$ a finite set of places of $K$ of cardinality $s$, $X$ a projective subvariety of $\P^N$ defined over $K$ of dimension $n\ge 1$ and of degree $d$, and $f^{(v)}_0,\ldots,f^{(v)}_m\ (v\in S)$ a systems of homogeneous polynomials in ${\overline{\Q}}[x_0,\ldots, x_N]$, where $m\ge n$. Assume that
\begin{align}\label{1.2}
\begin{split}
X({\overline{\Q}})\not\subset\{f^{(v)}_j=0\}\text{ for all }j=0,\ldots,m,v\in S,\\
X({\overline{\Q}})\cap\{f^{(v)}_0=0,\ldots,f^{(v)}_m=0\}=\emptyset\text{ for all }  v\in S
\end{split}
\end{align}
and assume also that for each $v\in S$, the distributive constant of the family of divisors $\{{\rm div} (f^{(v)}_0),$ $\ldots,{\rm div}  (f^{(v)}_m)\}$ with respect to $X$ does not exceed $\delta_X$, where $\delta_X>0$. Put $\alpha =\frac{(m+1)\delta_X}{n+\delta_X}$ and
\begin{align}
\label{1.3}
&\begin{cases}
&C :=\max\left([K(f^{(v)}_i ):K],v\in S,i=0,\ldots,m\right),\\
&\Delta := \mathrm{lcm}\left(\deg f^{(v)}_i\ :\ v\in S,i=0,\ldots,m\right),
\end{cases}\\
\label{1.4}
&\begin{cases}
&A_1:=\mathrm{exp}\left (2^{12n+4}(\alpha(n+1)+1)^{4n}\delta^{-2n}d^{2n+2}\Delta^{n(2n+2)}\right)\\
&\hspace{30pt} \times 4(m+1)s(2e(\alpha(n+1)+1)\delta^{-1})^{(m+1)s-1}\log (4C)\cdot\log\log (4C),\\
&A_2 :=(8n + 6)(\alpha(n+1)+1)^2d\Delta^{n+1}\delta^{-1},\\
&A_3 :=\mathrm{exp}\left (2^{6n+8}m(\alpha(n+1)+1)^{2n+2}\delta^{-n-1}d^{n+2}\Delta^{n(n+2)} \log(2Cs)\right ),\\
&H :=\log (N+m)+h(X)+\max\left(h(1,f^{(v)}_i), v\in S,0\le  i\le m\right).
\end{cases}
\end{align}
Then there are homogeneous polynomials $G_1,\ldots,G_u\in K[x_0,\ldots,x_N]$ with
$$u\le A_1, \deg G_i\le A_2\text{ for }i=1,\ldots,u$$
which do not vanish identically on $X$, such that the set of $\x\in X({\overline{\Q}})$ with
\begin{align}
\label{1.5}
&\log\left(\prod_{v\in S}\prod_{i=0}^m\max\limits_{\sigma\in G_K}\frac{|f^{(v)}_i(\sigma(\x))|_v^{1/\deg f^{(v)}_i}}{\|\sigma(\x)\|_v}\right)\le -\left(\alpha(n+1)+\delta\right)h(\x),\\
\label{1.6}
&h(\x)\ge A_3H
\end{align}
is contained in $\bigcup_{i=1}^u(X\cap\{G_i=0\})$.
\end{theorem}
We note that, if the family of divisors $\{{\rm div} (f^{(v)}_0),\ldots,{\rm div}  (f^{(v)}_m)\}$ is in $m$-subgeneral position with respect to $X$ then its distributive is bounded above by $(m-n+1)$. Hence, Theorem \ref{1.1} generalizes some previous results, such as \cite[Theorem 1.3]{EF08} of Evertse-Ferretti, \cite[Theorem 1.1]{LG} of Giang and \cite[Theorem 1.1]{Qmm} of Quang.

On the other hand, from Theorem A we may deduce the following version of the subspace theorem.

\noindent
{\bf Theorem B} (Evertse-Ferretti \cite{EF08}). {\it Let $X$ be a projective variety of dimension $n$ defined over a number field $K$. Let $S$ be a finite set of places of $K$. For each $v \in S$, let $D_{0, v}, \ldots, D_{n, v}$ be effective Cartier divisors on $X$, defined over $K$, in general position. Suppose that there exists an ample Cartier divisor $A$ on $X$ and positive integers $d_{j, v}$ such that $D_{j, v}$ is linearly equivalent to $d_{j, v} A$ (which we denote by $D_{j, v} \sim d_{j, v} A$ ) for $j=0, \ldots, n$ and all $v \in S$. Then, for every $\epsilon>0$, there exists a proper Zariski closed subset $Z \subset X$ such that for all points $x \in X(K) \backslash Z$,
$$\sum_{v \in S} \sum_{j=0}^{n}d^{-1}_{i,v}\lambda_{D_{j, v}, v}(x) \leq(n+1+\epsilon) h_{A}(x).$$} 
Here, by $\lambda_{D,v}$ we denote the Weil function of the divisor $D$ with respect to the place $v$. For this definition and detailed properties, we refer readers to \cite[Chapter 10]{Lan83} or \cite[Section 8]{Voj11}. For convenience, we list here some properties of Weil functions for varieties and divisors defined over a number field $K$. Let $X$ be a normal complete variety, let $D$ be a Cartier divisor on $X$, both defined over $k$. Then the following conditions are equivalent:
\begin{itemize}
\item[(1)] $D$ is effective.
\item[(2)] $\lambda_{D}$ is bounded from below by an $M_{K}$-constant.
\item[(3)] For all $v \in M_{K}, \lambda_{D, v}$ is bounded from below.
\item[(4)] There exists $v \in M_{K}$ such that $\lambda_{D, v}$ is bounded from below.
\end{itemize}
Now, let $Y$ be a closed subscheme of $X$. Then, there exist effective Cartier divisors $D_{1}, \cdots, D_{r}$ such that
$$Y=\bigcap_{i=1}^{r} D_{i}.$$
The (local) Weil function for $Y$ with respect to $v \in M_{k}$ is defined by
$$\lambda_{Y, v}=\min _{1 \leq i \leq r}\left\{\lambda_{D_{i}, v}\right\}.$$
We refer readers to the paper \cite{Sil87} of Silverman for fundamental properties of the Weil function.

We have the definition of the Seshadri constant of a closed subscheme with respect to a nef divisor as follows.

\noindent
{\bf Definition C.} {\it Let $Y$ be a closed subscheme of a projective variety $X$ and let $\pi: \tilde{X} \rightarrow X$ be the blowing-up of $X$ along $Y$. Let $A$ be a nef Cartier divisor on $X$. The Seshadri constant $\epsilon_{Y}(A)$ of $Y$ with respect to $A$ is defined by
$$\epsilon_{Y}(A)=\sup \left\{\gamma \in \mathbb{Q}^{\geq 0} \mid \pi^{*} A-\gamma E \text { is } \mathbb{Q} \text {-nef }\right\},$$
where $E$ is an effective Cartier divisor on $\tilde{X}$ whose associated invertible sheaf is the dual of $\pi^{-1} \mathcal{I}_{Y} \cdot \mathcal{O}_{\hat{X}}$.}

\begin{definition}\label{1.7} Let $X$ be a subvariety of a projective variety $V$ and $\mathcal Y=\{Y_{1}, \ldots, Y_{q}\}$ a family of closed subschemes on $V$ such that $X\not\subset Y_j\ (1\le j\le q)$. The distributive constant of the family $\mathcal Y=\{Y_{1}, \ldots, Y_{q}\}$ with respect to a subvariety $X\subset V$ is defined by
$$ \delta_{\mathcal Y,X}=\max_{\emptyset\ne J\subset\{1,\ldots,q\}}\max\left\{1;\frac{\sharp J}{\dim X-\dim X\cap\bigcap_{j\in J}Y_j}\right\},$$
where $\dim\emptyset=-\infty$.
\end{definition}

Recently, Heier and Levin generalized Theorem B to the case where the families of divisors are replaced by the families of closed subschemes in general position (see \cite[Theorem 1.3]{HL}). Our second aim in this paper is to generalize these results to the case of arbitrary families of closed subschemes. Our result is stated as follows.
\begin{theorem}\label{1.8}
 Let $X$ be a projective variety of dimension $n$ defined over a number field $k$. Let $S$ be a finite set of places of $k$. For each $v \in S$, let $Y_{0, v}, \ldots, Y_{q, v}$ be closed subschemes on $X$, defined over $k$ such that the distributive constant of the family $Y_{0, v}, \ldots, Y_{q, v}$ with respect to $X$ does not exceed $\delta_X$, where $\delta_X>0$. Let $A$ be an ample Cartier divisor on $X$. Then, for every $\epsilon>0$, there exists a proper Zariski closed subset $Z \subset X$ such that for all points $x \in X(k) \backslash Z$,
$$\sum_{v \in S} \sum_{j=0}^{q}\epsilon_{Y_{j,v}}(A)\lambda_{Y_{j, v}, v}(x) \leq(\delta_X(n+1)+\epsilon) h_{A}(x).$$
\end{theorem}
By simple computation, we see that if the family of closed subschemes $\{Y_{1,v},\ldots,Y_{q,v}\}$ is in $l$-subgeneral position with respect to $X$ then its distributive is bounded above by $(l-n+1)$. Therefore, Theorem \ref{1.8} also generalizes the recent result of He and Ru \cite[Main Theorem]{HR}.

\section{Quantitative Schmidt's subspace theorem}

In this section, we will prove Theorem \ref{1.1}. We need the following preparations.

\subsection{Chow weights}
We recall some notion and results from \cite{EveFer02,EF08}. 
Let $X$ be a projective subvariety of $\P^N$ defined over $K$ of dimension $n$ and degree $\Delta_X$. The Chow form of $X$ is the unique polynomial, up to a constant scalar, 
$$F_X(\textbf{u}_0,\ldots,\textbf{u}_n) = F_X(u_{00},\ldots,u_{0N};\ldots; u_{n0},\ldots,u_{nN})$$
in $N+1$ blocks of variables $\textbf{u}_i=(u_{i0},\ldots,u_{iN}), i = 0,\ldots,n$ with the following
properties: 
\begin{itemize}
\item $F_X$ is irreducible in $K[u_{00},\ldots,u_{nN}]$;
\item $F_X$ is homogeneous of degree $\delta$ in each block $\textbf{u}_i, i=0,\ldots,n$;
\item $F_X(\textbf{u}_0,\ldots,\textbf{u}_n) = 0$ if and only if $X\cap H_{\textbf{u}_0}\cap\cdots\cap H_{\textbf{u}_n}$ contains a $\overline{K}$-rational point, where $H_{\textbf{u}_i}, i = 0,\ldots,n$, is the hyperplane given by $u_{i0}x_0+\cdots+ u_{iN}x_N=0.$
\end{itemize}
Let ${\bf c}=(c_0,\ldots, c_N)$ be a tuple of real numbers and let $t$ be an auxiliary variable. We write
\begin{align*}
F_X(t^{c_0}u_{00},&\ldots,t^{c_N}u_{0N};\ldots ; t^{c_0}u_{n0},\ldots,t^{c_N}u_{nN})\\ 
& = t^{e_0}G_0(\textbf{u}_0,\ldots,\textbf{u}_N)+\cdots +t^{e_r}G_r(\textbf{u}_0,\ldots, \textbf{u}_N),
\end{align*}
with $G_0,\ldots,G_r\in K[u_{00},\ldots,u_{0N};\ldots; u_{n0},\ldots,u_{nN}]$ and $e_0>e_1>\cdots>e_r$. The Chow weight of $X$ with respect to ${\bf c}$ is defined by
\begin{align*}
e_X({\bf c}):=e_0.
\end{align*}

\begin{theorem}[{see \cite[Theorem 3.2]{Qjnt}}]\label{2.1}
Let $Y$ be a projective subvariety of $\P^R$ defined over $K$ of dimension $n\ge 1$ and degree $\Delta_Y$. Let $m\ (m\ge n)$ be an integer and let ${\bf c}=(c_0,\ldots,c_R)$ be a tuple of non-negative reals. Let $H_0,\ldots,H_R$ be hyperplanes in $\P^R$ defined by $H_{i}=\{y_{i}=0\}\ (0\le i\le R)$. Let $\{i_0,\ldots, i_m\}$ be a subset of $\{0,\ldots,R\}$ such that:
\begin{itemize}
\item[(1)] $c_{i_m}=\min\{c_{i_0},\ldots,c_{i_m}\}$,
\item[(2)] $Y(\overline{K})\cap\bigcap_{j=0}^{m-1}H_{i_j}\ne\emptyset$, 
\item[(3)] and $Y\not\subset H_{i_j}$ for all $j=0,\ldots,m$.
\end{itemize}
Let $\delta_Y$ be the distributive constant of the family $\{H_{i_j}\}_{j=0}^m$ with respect to $Y$. Then
$$e_Y({\bf c})\ge \frac{\Delta_Y}{\delta_Y}(c_{i_0}+\cdots+c_{i_m}).$$
\end{theorem}

Applying the above theorem, we have the following theorem.

\begin{theorem}\label{2.2}
Let $Y$ be a projective subvariety of $\P^R$ defined over $K$ of dimension $n\ge 1$ and degree $\Delta_Y$. Let $m\ (m\ge n)$ be an integer and let ${\bf c}=(c_0,\ldots,c_R)$ be a tuple of non-negative reals. Let $H_0,\ldots,H_R$ be hyperplanes in $\P^R$ defined by $H_{i}=\{y_{i}=0\}\ (0\le i\le R)$. Let $\{i_0,\ldots, i_m\}$ be a subset of $\{0,\ldots,R\}$ such that:
$$Y(\overline{K})\cap\bigcap_{j=0}^{m}H_{i_j}=\emptyset\text{  and }Y\not\subset H_{i_j}\text{ for all }j=0,\ldots,m.$$
Assume that the family $\{H_{i_j}\}_{j=0}^m$ has the distributive constant with respect to $Y$ at most $\delta$. Then
$$e_Y({\bf c})\ge \frac{\Delta_Y(n+\delta)}{(m+1)\delta}(c_{i_0}+\cdots+c_{i_m}).$$
\end{theorem}
\begin{proof}
Without loss of generality, we may assume that $c_{i_0}\ge\cdots\ge c_{i_m}$ and $Y(\overline{K})\cap\bigcap_{j=0}^{l-1}H_{i_j}\ne\emptyset, Y(\overline{K})\cap\bigcap_{j=0}^{l}H_{i_j}=\emptyset$ for an index $l,\ 1\le l\le m$.
Let $\delta'$ be the distributive constant of the family $\{H_{i_j}\}_{j=0}^l$ with respect to $Y$. Then by Theorem \ref{2.1}, we have
\begin{align*}
e_Y({\bf c})&\ge \frac{\Delta_Y}{\delta'}(c_{i_0}+\cdots+c_{i_l})\\
&\ge \frac{\Delta_Y(l+1)}{(m+1)\delta'}(c_{i_0}+\cdots+c_{i_m})\\
[\text{since }\delta'\le l-n+1]\ \ &\ge \frac{\Delta_Y(n+\delta')}{(m+1)\delta'}(c_{i_0}+\cdots+c_{i_m})\\
[\text{since }\delta'\le \delta]\ \ &\ge \frac{\Delta_Y(n+\delta)}{(m+1)\delta}(c_{i_0}+\cdots+c_{i_m}).
\end{align*}
The theorem is proved.
\end{proof}

\subsection{Twisted height}
Let $K$ be a number field and $L$ be a finite extension of $K$. If $w$ is a place of $L$ which lies above a place $v$ of $K$, then
$$|x|_w=|x|^{d(w|v)}_v, \text{ for } x\in K\text{ with } d(w|v)=\frac{[L_w : K_v]}{[L:K]}.$$
For $v\in M_K$, let $\c_v=(c_{0v},\ldots, c_{Rv})$ be a tuple of reals such that $c_{0v}=\cdots=c_{Rv}=0$ for all but finitely many places $v\in M_K$ and put $\c=(\c_v\ :\ v\in M_K)$. Further, let $Q\ge 1$ be a real. 
For $\y=(y_0,\ldots,y_R)\in\P^R(K)$, we define the twisted height of $\y$ by
$$H_{Q,\c}(\y):=\prod_{v\in M_K}\max\limits_{0\le i\le R}\left (|y_i|_vQ^{c_{iv}}\right).$$

For any finite extension $L$ of $K$ we put
$$c_{iw}:=c_{iv}.d(w|v) \text{ for } w\in M_L,$$
where $M_L$ is the set of places of $L$ and $v$ is the place of $K$ lying below $w.$ The twisted height of $\y\in\P^R({\overline{\Q}})$ is defined by
\begin{align}\label{2.3}
H_{Q,\c}(\y):=\prod_{w\in M_L}\max\limits_{0\le i\le R}\left(|y_i|_wQ^{c_{iw}}\right),
\end{align}
where $L$ is any finite extension of $K$ such that $\y\in \P^R(L)$. 

\subsection{Some auxiliary results} Let $Y$ be a projective subvariety of $\P^R$ of dimension $n\ge 1$ and degree $D$, defined over $K$, and let $c_v=(c_{0v},\ldots,c_{Rv})\ (v\in M_K)$ be tuples of reals such that
\begin{align}
\label{2.4}
&c_{iv}\ge 0\text{ for }v\in M_K,i=0,\ldots , R,\\ 
\label{2.5}
&c_{0v}=\cdots=c_{Rv}=0 \text{ for all but finitely many }v\in M_K,\\
\label{2.6}
&\sum_{v\in M_K}\max\{c_{0,v},\ldots,c_{R,v}\}\le 1. 
\end{align}
Put
\begin{align}\label{2.7}
E_Y({\bf c}):=\frac{1}{(n+1)D}\left (\sum_{v\in M_K}e_Y({\bf c}_v)\right),
\end{align}
where $e_Y(\c_v)$ is the Chow weight defined in Section 2.

For $0<\delta\le 1$, we put
\begin{align}\label{2.8}
\begin{split}
\begin{cases}
B_1&:=\mathrm{exp}\left (2^{10n+4}\delta^{-2n}D^{2n+2}\right).\log (4R)\log\log (4R),\\
B_2&:=(4n+3)D\delta^{-1},\\
B_3&:=\exp\left(2^{5n+4}\delta^{-n-1}D^{n+2}\log (4R)\right).
\end{cases}
\end{split}
\end{align}
\begin{theorem}[{see  \cite[Theorem 2.1]{EF08}}]\label{2.9}
There are homogeneous polynomials $F_1,\ldots,F_t\in K[y_0,\ldots,y_R]$ with
$$t\le  B_1, \deg F_i\le  B_2 \text{ for }i = 1,\ldots,t,$$
which do not vanish identically on $Y,$ such that for every real number $Q$ with
$$\log Q \ge B_3.(h(Y)+1)$$
there is $F_i\in\{F_1,\ldots,F_t\}$ with
$$\left\{y\in Y({\overline{\Q}})\ :\ H_{Q,\c}({\bf y})\le Q^{E_Y({\bf c})-\delta}\right\}\subset Y\cap\{F_i=0\}.$$
\end{theorem}

For a polynomial $f$, we write $f=\sum_{m\in M_f}c_f(m)m$, where the symbol $m$ denotes a monomial, $M_f$ is a finite set of monomials, and $c_f(m)\ (m\in M_f)$ are the coefficients. For any map $\sigma$ on the definition field of $f$, we put
$$ \sigma (f):=\sum_{m\in M_f}\sigma (c_f(m))\cdot m .$$
Let $f_i=\sum_{m\in M_{f_i}}c_{f_i}(m)\cdot m\ (i=1,\ldots,r)$ be $r$ polynomials with complex coefficients. We define the following norms:
$$\|f_1,\ldots,f_r\|:=\max\left (|c_{f_i}(m)|\ : 1\le i \le r,m\in M_{f_i}\right ),$$
$$\|f_1,\ldots, f_r\|_1:=\sum_{i=1}^r\sum_{m\in M_{f_i}}|c_{f_i}(m)|.$$
If all coefficients of $f_i\ (i=1,...,r)$ belong to ${\overline{\Q}}$, we define
\begin{align}\label{2.10}
\begin{split}
\begin{cases}
&\|f_1,\ldots, f_r\|_v:=\max\left(|c_{f_i}(m)|_v\ :\ 1\le i \le r,m\in M_{f_i}\right)\ (v\in M_K),\\
&\|f_1,\ldots, f_r\|_{v,1}:=\|f_1,\ldots ,f_r\|_v\ (v\in M^0_K),\\
&\|f_1,\ldots , f_r\|_{v,1}:= \|\sigma_v(f_1),\ldots ,\sigma_v(f_r)\|_1^{[K_v:R]/[K:Q]}\  (v\in M^\infty_K),
\end{cases}
\end{split}
\end{align}
where $\sigma_v$ (for $v\in M^\infty_K$) is the isomorphic embedding $\sigma_v:{\overline{\Q}}\rightarrow\C$ such that
$$|x|_v=|\sigma_v(x)|^{[K_v:\R]/[K:\Q]}\text{ for }x\in{\overline{\Q}}.$$
If all coefficients of $f_i\ (i=1,...,r)$ belong to $K$, we may define heights
$$h(f_1,\ldots, f_r):=\log\left(\prod_{v\in M_K}\|f_1,\ldots, f_r\|_v\right ),$$
$$h_1(f_1,\ldots, f_r):=\log\left(\prod_{v\in M_K}\|f_1,\ldots, f_r\|_{v,1}\right ).$$
More generally, for polynomials $f_1,\ldots,f_r$ with coefficients in ${\overline{\Q}}$, we  choose a number field $K$ containing the coefficients of $f_1,\ldots,f_r$ and define the weights $h(f_1,\ldots,f_r)$, $h_1(f_1,\ldots,f_r)$ as above. We see that these definitions are independent of the choice of $K$.

\begin{lemma}[{see \cite[Lemma 5.2]{EF08}}]\label{2.11}
Let $\theta$ be a real with $0<\theta\le\frac{1}{2}$ and let $q$ be a positive integer. Then there exists a set $W$ of cardinality at most $(e/\theta)^{q-1}$, consisting of tuples $(c_1,\ldots,c_q)$ of nonnegative reals with $c_1+\cdots+c_q=1$, with the following property: for every set of reals $A_1,\ldots,A_q$ and $\Lambda$ with $A_j\le 0$ for $j=1,\ldots,q$ and $\sum_{j=1}^qA_j\le -\Lambda$, there exists a tuple $(c_1,\ldots,c_q)\in W$ such that
$$A_j\le -c_j(1-\theta)\Lambda\text{ for }j=1,\ldots,q.$$
\end{lemma}

\vskip0.2cm
\noindent
\subsection{Proof of Quantitative Subspace Theorem \ref{1.1}} 

In this subsection, we will prove Theorem \ref{1.1}. We mainly follow the method of Evertse and Ferretti in \cite{EF08}. 

\vskip0.1cm
\textbf{(a)} let $K$ be a number field, $S$ be a finite set of places of $K$ and $X,N,m,n,d,s,C,f^{(v)}_i$ $(v\in S,i=0,\ldots,m),C,\Delta,A_1,A_2,A_3,H$ be as in Theorem \ref{1.1}. Denote the coordinates on $\P^N$ by ${\bf x}=(x_0,\ldots,x_N).$

Denote by $f_0,\ldots,f_R$ the distinct polynomials among $\sigma(f_{j}^{(v)})\ (v\in S,j=0,\ldots,m,\sigma\in G_K)$. From (\ref{1.3}), we have
\begin{align}\label{2.12}
R\le C(m+1)s-1.
\end{align}
Put $g_i=f_i^{\Delta/\deg f_i}$ for $i=0,\ldots, R$. Then $g_0,\ldots,g_R$  are homogeneous polynomials in $K[x_0,\ldots,x_u]$ of degree $\Delta$. Consider the map
$$\varphi : {\bf x}\mapsto (g_0({\bf x}),\ldots,g_R({\bf x})), Y:=\varphi(X).$$
By (\ref{1.2}), $\varphi$ is a finite morphism on $X$, and $Y$ is a projective subvariety of $\P^R$ defined over $K'$, where $K'$ is the extension of $K$ generated by all coefficients of $f_0,\ldots,f_R$. . We have
\begin{align}\label{2.13}
\dim Y=n, \deg Y=: D\le d\Delta^n.
\end{align}
For every $v'\in M_{K'}$, we choose an extension of $|.|_{v'}$ to ${\overline{\Q}}$. Since $K'/K$ is a normal extension, for every $v'\in M_{K'}$, there is $\tau_{v'}\in G_K$ such that
\begin{align}\label{2.14}
|x|_{v'} =|\tau_{v'}(x)|^{1/g(v)}_v\text{ for }x\in{\overline{\Q}},
\end{align}
where $v\in M_K$ is the place below $v'$ and $g(v)$ is the number of places of $K'$ lying above $v$. For each $v'\in M^\infty_{K'}$, there is an isomorphic embedding $\sigma_{v'}:K'\hookrightarrow\C$ such that $|x|_{v'}=|\sigma_{v'}(x)|^{[K_{v'}:\R]/[K':\Q]}$ for $x\in{\overline{\Q}}$. Then, we define norms $\|.\|_{v'}$,$\|.\|_{v',1}$ for polynomials similarly as in (\ref{2.10}), with $K',v',\sigma_{v'}$ in place of $K,v,\sigma_v.$

\vskip0.2cm
\noindent
\vskip0.1cm
\textbf{(b)} From \cite[Equations (3.19), (3.20)]{Qmm}, we have the following estimate: 
\begin{align}\label{3.15}
h_1(1,g_0,\ldots,g_R)&\le  6\Delta^2 Cms\cdot H;\\
\label{2.16}
h(Y)&\le 25nmd\Delta^{n+2}Cs\cdot H,
\end{align}
where $H$ is defined by (\ref{1.4}). 

\vskip0.1cm
\textbf{(c)}
Fix a solution ${\bf x}\in X({\overline{\Q}})$ of (\ref{1.5}). For $v\in S$, denote by $I_v$ the subset of $\{0,\ldots,R\}$ such that $\{f^{(v)}_j : j=0,\ldots,m\}=\{f_i : i\in I_v\}$. Put $G_v:=\|1,g_0,\ldots,g_R\|_{v,1}$ for $v\in S.$ Then
$$\sum_{v\in S}\sum_{i\in I_v}\log\left (\max\limits_{\sigma\in G_K}\frac{|g_i({\bf x})|_v}{G_v\|\sigma({\bf x})\|^{\Delta}_v}\right )\le -(\alpha(n+1)+\delta)\Delta h({\bf x}).$$
We see that all terms in the sum are $\le 0$. Applying Lemma \ref{2.11} with $q=(m+1)s$ and $\theta=\frac{\delta}{2(\alpha(n+1)+\delta)}=1-\frac{\alpha(n+1)+\delta/2}{\alpha(n+1)+\delta}$, we get a set $W$ with
\begin{align}\label{2.17}
\sharp W\le\left (\frac{2e(\alpha(n+1)+\delta)}{\delta}\right)^{(m+1)s-1}&\le \left[(2e(\alpha(n+1)+1)\delta^{-1})^{(m+1)s-1}\right].
\end{align}
consisting of tuples of nonnegative reals $(c_{iv}\ :\ v\in S, i\in I_v)$ with
\begin{align}\label{2.18}
\sum_{v\in S}\sum_{i\in I_v}c_{iv}=1
\end{align}
such that for every solution ${\bf x}\in X({\overline{\Q}})$ of (\ref{1.5}) there is a tuple $(c_{iv}: v\in S, i\in I_v)\in W$ with
\begin{align}\label{2.19}
\log\left(\max\limits_{\sigma\in G_K}\frac{|g_i(\sigma({\bf x}))|_v}{G_v\cdot \|\sigma({\bf x})\|^{\Delta}_v}\right)\le -c_{iv}\left (\alpha(n+1)+\frac{\delta}{2}\right)\Delta h({\bf x})
\end{align}
for all $v\in S, i\in I_v.$ 

Denote by $S'$ the set of places of $K'$ lying above the places in $S$. Let $v'\in S'$ and let $v$ be the place of $K$ lying below $v'$ and $\tau_{v'}\in G_K$ be given by (\ref{2.14}). We define the subset $I_{v'}\subset\{0,\ldots,R\}$ and $c_{i,v'}\ (i\in I_{v'})$ by
\begin{align*}
&\{g_i : i\in I_{v'}\}=\{\tau^{-1}_{v'}(g_j)\ :\ j\in I_v\}\text{ for }v'\in S',\\ 
&c_{i,v'}:=\frac{c_{jv}}{g(v)} \text{ for }v'\in S', i\in I_{v'},
\end{align*}
where $j\in I_v$ is the index such that $g_i=\tau^{-1}_{v'}(g_j)$ and $g(v)$ is the number of places of $K'$ lying above $v$. Further, we put
$$ G_{v'}:=\|1,g_0,\ldots,g_R\|_{v',1}\text{ for }v'\in M_{K'}. $$
By (\ref{2.14}), we may rewrite (\ref{2.19}) as
\begin{align}\label{2.20}
\log\left (\max_{\sigma\in G_K}\frac{|g_i(\sigma({\bf x}))|_{v'}}{G_{v'}\cdot \|\sigma({\bf x})\|^{\Delta}_{v'}}\right )\le -c_{i,v'}\left (\alpha(n+1)+\frac{\delta}{2}\right)\Delta h({\bf x})
\end{align}
for all $v'\in S',i\in I_{v'}.$ Combining (\ref{2.17}), (\ref{2.18}) we obtain the following lemma. 
\begin{lemma}\label{2.21}
There is a set $W'$ of cardinality at most $\left[(2e(\alpha(n+1)+1)\delta^{-1})^{(m+1)s-1}\right]$, consisting of tuples of nonnegative reals $(c_{i,v'} : v'\in S', i\in I_{v'})$ with
\begin{align}\label{2.22}
\sum_{v\in S'}\sum_{i\in I_{v'}}c_{i,v'}=1,
\end{align}
with the property that for every ${\bf x}\in X({\overline{\Q}})$ with (\ref{1.5}) there is a tuple in $W'$ such that ${\bf x}$ satisfies (\ref{2.20}).
\end{lemma}
Consider the solutions of a fixed system (\ref{2.20}). Put
\begin{align}\label{2.23}
\begin{split}
&\text{$\bullet$ $c_{i,v'}=0$ for $v'\in S',i\in\{0,\ldots,R\}\setminus I_{v'}$ and $v'\in M_{K'}\setminus S', i=0,\ldots,R$.}\\
&\text{$\bullet$ $\c_{v'}:=(c_{0,v'},\ldots,c_{R,v'})$ for $v'\in M_{K'}$ and ${\bf c}:=({\bf c}_{v'}:\ v'\in M_{K'})$.}
\end{split}
\end{align}
\noindent
Denote by ${\bf y}=(y_0,\ldots,y_R)$ the coordinates of $\P^R$. We define $H_{Q,{\bf c}}({\bf y}), E_Y({\bf c})$ similarly as (\ref{2.3}), (\ref{2.7}), respectively, but with $K'$ in place of $K$.

\begin{lemma}\label{2.24}
Let ${\bf x}\in X({\overline{\Q}})$ be a solution of (\ref{2.20}) satisfying (\ref{1.6}) and let $\sigma\in G_K$. Put
$${\bf y}:=\varphi(\sigma({\bf x})), Q:=\mathrm{exp}\left(\left(\alpha(n+1)+\delta/2\right)\Delta h({\bf x})\right).$$
Then
\begin{align}\label{3.31}
H_{Q,{\bf c}}({\bf y})\le Q^{E_Y({\bf c})-\delta/2(\alpha(n+1)+1)^2}.
\end{align}
\end{lemma}
\begin{proof}
Let $v'\in S'$ and $I_{v'}=\{i_0,...,i_m\}$. The assumption (\ref{1.2}) implies that 
$$ X(\overline{\Q})\cap\{g_{i_0}=0,...,g_{i_m}=0\}=\emptyset. $$
Since $Y=\varphi(X)$, for every $\y =(y_0,...,y_R)\in Y(\overline{\Q})$, there is $\x=(x_0,...,x_N)\in X(\overline{\Q})$ such that $y_i=g_i(\x)\ (0\le i\le R)$. Therefore,
$$ Y(\overline{\Q})\cap\{y_{i_0}=0,....,y_{i_m}=0\}=\emptyset. $$
By Theorem \ref{2.2}, we have
\begin{align*}
\frac{1}{(n+1)D}e_Y(\c_{v'})\ge\dfrac{1}{\alpha(n+1)}(c_{i_0,v'}+\cdots+c_{i_m,v'})=\frac{1}{\alpha(n+1)}\sum_{i\in I_{v'}}c_{i,v'}.
\end{align*}
We also note that if $v''\not\in S'$, we have $c_{i,v''}=0\ (0\le i\le R)$ and hence $e_{Y}(\c_{v''})=0$. Therefore, summing up the both sides of the above inequality over all $v'\in S'$, we obtain. 
\begin{align}\label{2.26}
E_Y(\c)\ge\frac{1}{\alpha(n+1)},\ \ \text{(by (\ref{2.22}))}. 
\end{align}

Now, consider a solution $\x$ of (\ref{2.20}) which satisfies (\ref{1.6}) and an element $\sigma\in G_K$. From (\ref{2.23}), we see that $\sigma(\x)$ satisfies (\ref{2.20}) for all $v\in M_K$ and $i=0,...,R$. Put $y_i=g_i(\sigma(\x))\ (i=0,...,R)$ and $\y=(y_0,...,y_R)=\varphi(\sigma(\x))$. Choose a finite normal extension $L$ of $K'$ such that $\sigma(\x)\in X(L)$. Let $\omega$ be a place of $L$. We take $v'$ to be the place of $K'$ lying below $\omega$. Then there exists $\tau_\omega\in G_{K'}$ such that $|x|_\omega=|\tau_{\omega}(x)|_{v'}^{d(\omega|v')}$ for $x\in L$, where $d(\omega|v')=[L_{\omega}:K'_{v'}]/[L:K']$. Putting $c_{i\omega}=d(\omega|v')c_{i,v'}$, we have the following estimate
\begin{align*}
|y_i|_\omega Q^{c_{i\omega}}&=|g_i(\sigma(\x))|_\omega Q^{c_{i\omega}}=\left(|g_i(\tau_\omega(\sigma(\x)))|_{v'}Q^{c_{i,v'}}\right )^{d(\omega|v')}\\
&\le \left(G_{v'}\|\tau_\omega(\sigma(\x))\|_{v'}^{\Delta}\right)^{d(\omega|v')}= G_{v'}^{d(\omega|v')}\|\sigma(\x)\|_{\omega}^{\Delta}.
\end{align*}
This implies that
$$ \max_{0\le i\le R}|y_i|_\omega Q^{c_{i\omega}}\le G_{v'}^{d(\omega|v')}\|\sigma(\x)\|_{\omega}^{\Delta}.$$
Taking the product of both sides of the above inequality over all $\omega\in M_L$ and using $h(\sigma(\x))=h(\x)$, we obtain
\begin{align}\label{2.27}
H_{Q,\c}(\y)\le \exp(h_1(1,g_0,...,g_R))Q^{1/(\alpha(n+1)+\delta/2)}.
\end{align}
By the definition of $Q$ and the inequality (\ref{2.26}), we have
\begin{align*}
\bigl (E_Y(\c)-&\frac{\delta}{2(\alpha(n+1) +1)^2}-\frac{1}{\alpha(n+1)+\delta/2}\bigl)\log Q\\
&\ge\left( \frac{1}{\alpha(n+1)}-\frac{\delta}{2(\alpha(n+1) +1)^2}-\frac{1}{\alpha(n+1)+\delta/2}\right)\left(\alpha(n+1)+\frac{\delta}2\right)\Delta h(\x)\\
&=\frac{\delta\left(2\alpha(n+1)-\alpha(n+1)\delta/2+1\right)}{2\alpha(n+1)(\alpha(n+1)+1)^2}\Delta h(\x)\ge\frac{3\delta}{4(\alpha(n+1)+1)^2}\Delta h(\x)\\
\text{[by using (\ref{1.6})]}\ \ & \ge\frac{3\delta\Delta}{4(\alpha(n+1)+1)^2}A_3H\\
\text{[from (\ref{3.15})]}\ \ &\ge 6\Delta^2Cms H\ge h_1(1,g_0,...,g_R).
\end{align*}
Combining the above inequality and (\ref{2.27}) we obtain
$$ H_{Q,\c}(\y)\le Q^{E_Y(\c)-\frac{\delta}{2(\alpha(n+1) +1)^2}}. $$
The lemma is proved.
\end{proof}

\begin{proof}[{\bf (d)} Proof of Theorem \ref{1.1}]
We have the following fundamental estimate:
\begin{align*}
2^{6n+7}m&(\alpha(n+1)+1)^{2n+2}\delta^{-n-1}d^{n+2}\Delta^{n(n+2)} \log(2Cs) \\ 
&\ge 2^{5n+4}2^{n+1}(\alpha(n+1)+1)^{2n+2}\delta^{-n-1}d^{n+2}\Delta^{n(n+2)}\log (2^3Cms)\\
&\ge 2^{5n+4}(2(\alpha(n+1) +1)^2\delta^{-1})^{n+1}d^{n+2}\Delta^{n(n+2)}\log (4C(m+1)s);\\
\text{ and }\ 2^{6n+7}m&(\alpha(n+1)+1)^{2n+2}\delta^{-n-1}d^{n+2}\Delta^{n(n+2)} \log(2Cs)\\ 
&\ge 26nmd\Delta^{n+2}\log (2Cs)\ge\log (26nmd\Delta^{n+2}Cs). 
\end{align*}
Denote by $B'_1,B'_2,B'_3$ the quantities obtained by substituting $\delta/(2(\alpha(n+1)+1))^2$ for $\delta, C(m+1)s-1$ for $R$, and $d\Delta^n$ for $D$ in the quantities $B_1,B_2,B_3$ in (\ref{2.8}). 
Then, for ${\bf x}$ satisfying (\ref{1.6}) and $Q=\exp{\left (\alpha(n+1)+\frac{\delta}{2}\right )\Delta h({\bf x})}$, we have
\begin{align*}
\log Q&=\left (\alpha(n+1)+\frac{\delta}{2}\right )\Delta h({\bf x})\ge A_3H\\
&=\mathrm{exp}\left (2^{6n+8}m(\alpha(n+1)+1)^{2n+2}\delta^{-n-1}d^{n+2}\Delta^{n(n+2)} \log(2Cs)\right )\cdot H\\
&\ge\mathrm{exp}\left (2^{5n+4}(2(\alpha(n+1)+1)^2\delta^{-1})^{n+1}d^{n+2}\Delta^{n(n+2)}\log (4C(m+1)s)\right )\\
&\ \ \ \times \left (26nmd\Delta^{n+2}Cs\right )\cdot H\\
&=B'_3\cdot\left(26nmd\Delta^{n+2}Cs\right )\cdot H\\
\text{[by (\ref{2.16})]}\ &\ge B'_3(h(Y)+1).
\end{align*}
Moreover, if ${\bf x}$ is also a solution of (\ref{1.5}) then by applying Lemma \ref{2.24} for the points ${\bf y}=\varphi(\sigma({\bf x}))$ $(\sigma\in G_K)$, we have $H_{Q,{\bf c}}({\bf y})\le Q^{E_Y({\bf c})-\delta/2(\alpha(n+1)+1)^2}$. Now we apply Theorem \ref{2.9} with $K',\delta/(2(\alpha(n+1)+1))^2$ in place of $K,\delta$ and, in view of (\ref{2.12}) and (\ref{2.13}), with $D\le d\Delta^n$ and $R=C(m+1)s-1.$ From (\ref{2.22}), (\ref{2.23}), we see that the conditions (\ref{2.4}), (\ref{2.5}), (\ref{2.6}) (with $K'$ in place of $K$) are satisfied. Therefore, Theorem \ref{2.9} implies that there are homogeneous polynomials $F_1,\ldots,F_t$ $\in K'[y_0,\ldots,y_R]$ not vanishing identically on $Y$, with $t\le B_1'$ and $\deg F_i\le  B'_2$ for $i=1,\ldots,t$, so that: for such above points ${\bf y}$, there is $F_i\in\{F_l,\ldots,F_t\}$ such that $F_i({\bf y})=0$. And hence, for every solution ${\bf x}$ of (\ref{1.5}) with (\ref{1.6}), there is $F_i\in\{F_l,\ldots,F_t\}$ such that $F_i(\varphi(\sigma({\bf x})))=0$ for every $\sigma\in G_K$.

Therefore $\tilde F_i(\sigma({\bf x}))=0$ for every $\sigma\in G_K$, where $\tilde F_i$ is the polynomial obtained by substituting $g_j$ for $y_j$ in $F_i$ for $j=0,\ldots,R$. We note that $\tilde F_i\in K'[x_0,\ldots,x_N],\deg\tilde F_i\le B'_2\Delta$. Then we may write $\tilde F_i=\sum_{k=1}^M\omega_k\tilde F_{ik}$, where $\omega_1,\ldots,\omega_M$ is a $K$-basis of $K'$, and the $\tilde F_{ik}$ are polynomials with coefficients in $K$. Since $\tilde F_i$ does not vanish identically on $X$, we may choose a polynomial $G_i$ among $\{\tilde F_{ik}: k=1,\ldots,M\}$ not vanishing identically on $X$. Since $\sigma (\tilde F_i)({\bf x})=0$ for $\sigma\in G_K$ and the polynomials $\tilde F_{ik}$ are linear combinations of the polynomials $\sigma(\tilde F_i)\ (\sigma\in G_K)$, one has $\tilde F_{ik}({\bf x})=0$ for $k=1,\ldots,M$, so in particular $G_i({\bf x})=0$.

It means that there are homogeneous polynomials $G_1,\ldots,G_t\in K[x_0,\ldots,x_N]$ with $t\le B'_1$ and $\deg G_i\le B'_2\Delta$ for $i=1,\ldots,t$ not vanishing identically on $X$, such that the set of ${\bf x}\in X({\overline{\Q}})$ satisfying (\ref{2.20}) and (\ref{1.6}) is contained in $\bigcup_{i=1}^t(X\cap\{G_i=0\})$.

By Lemma \ref{2.21}, there are at most $T:=\left[(2e(\alpha(n+1)+1)\delta^{-1})^{(m+1)s-1}\right]$ different systems (\ref{2.20}), such that every solution ${\bf x}\in X({\overline{\Q}})$ of (\ref{1.5}) satisfies one of these systems. Therefore, there are homogeneous polynomials $G_1,\ldots,G_u\in K[x_0,\ldots,x_N]$ not vanishing identically on $X$, with $u\le B'_1T$ and with $\deg G_i\le B'_2\Delta$ for $i=1,\ldots,u,$ such that all solutions ${\bf x}\in X({\overline{\Q}})$ of (\ref{1.5}) with (\ref{1.6}) are contained in $\bigcup_{i=1}^u(X\cap\{G_i=0\}).$ 

In order to complete the proof of Theorem \ref{1.1}, it remains to show that $B'_2\Delta=A_2$ and $B'_1T\le A_1$. Indeed, we have:
\begin{align}\nonumber
\bullet\ \  & B'_2\Delta=(4n+3)(d\Delta^n)(2(\alpha(n+1)+1)^2\delta^{-1})\Delta\\
\nonumber
&\hspace{20pt}=(8n+6)(\alpha(n+1)+1)^2d\Delta^{n+1}\delta^{-1}=A_2,\\
\nonumber
\bullet\ \  &B'_1T\le\mathrm{exp}\left (2^{10n+4}(2(\alpha(n+1)+1)^2)^{2n}\delta^{-2n}(d\Delta^n)^{2n+2}\right)\\
\label{2.28}
&\hspace{20pt}\times\log(4(m+1)Cs)\log\log(4(m+1)Cs)\cdot (2e(\alpha(n+1)+1)\delta^{-1})^{(m+1)s-1}.
\end{align}
Note that $\log (\lambda x)\le\sqrt{2x}\log(\lambda)$ for all $x\ge 2, \lambda\ge 4$ and $\log (\lambda x)\le 2x\log(\lambda)$ for all $x\ge 2, \lambda\ge\log 4$.
Then, we have the following fundamental estimates:
\begin{align*}
\log(4(m+1)Cs)\cdot\log\log(4(m+1)Cs)&\le \sqrt{2(m+1)s}\log (4C)\cdot \log(\sqrt{2(m+1)s}\log (4C))\\
&\le 4(m+1)s\log (4C)\cdot\log\log (4C).
\end{align*}
Therefore, from (\ref{2.28}) we have
\begin{align*}B'_1T\le&\mathrm{exp}\left (2^{12n+4}(\alpha(n+1)+1)^{4n}\delta^{-2n}d^{2n+2}\Delta^{n(2n+2)}\right)\\
&\times 4(m+1)s(2e(\alpha(n+1)+1)\delta^{-1})^{(m+1)s-1}\log (4C)\cdot\log\log (4C)=A_1.
\end{align*}
The proof of the theorem is completed.
\end{proof}

\section{Subspace theorem for arbitrary closed schemes}

In order to prove Theorem \ref{1.8}, we need the following lemmas.

\begin{lemma}[{see \cite[Lemma 38]{Qjge} and also \cite[Lemma 3.2]{Qpcf}}]\label{3.1}
Let $V$ be a projective subvariety of $\P^N(\C)$ of dimension $n$. Let $Q_0,\ldots,Q_{l}$ be $l$ hypersurfaces in $\P^N(\C)$ with defining homogeneous polynomials $\tilde{Q}_0,\ldots,\tilde{Q}_{l}$ of the same degree $d\ge 1$, such that $\bigcap_{i=0}^{l}Q_i\cap V=\emptyset$ and
$$\dim\left (\bigcap_{i=0}^{s}Q_i\right )\cap V=n-u\ \text{ for } u\in\{1,\ldots,n\}, s\in\{t_{u-1},\ldots,t_u\},$$
where $t_0,t_1,\ldots,t_n$ are integers with $0=t_0<t_1<\cdots<t_n=l$. Then there exist $n+1$ hypersurfaces $P_0,\ldots,P_n$ in $\P^N(\C)$ with the defining homogeneous polynomials of the forms
$$\tilde{P}_u=\sum_{j=0}^{t_{u}}c_{uj}\tilde{Q}_j, \ c_{uj}\in\C,\ u=0,\ldots,n,$$
such that $\left (\bigcap_{u=0}^{n}P_u\right )\cap V=\emptyset.$
\end{lemma}

\begin{lemma}[{see \cite[Lemma 3.9]{Qjge} also \cite[Lemma 3.1]{Qpcf}}]\label{3.2}
Let $t_0,t_1,\ldots,t_n$ be $n+1$ integers such that $1=t_0<t_1<\cdots <t_n$, and let $\delta =\underset{1\le s\le n}\max\dfrac{t_s-t_0}{s}$. 
Then for every $n$ real numbers $a_0,a_1,\ldots,a_{n-1}$  with $a_0\ge a_1\ge\cdots\ge a_{n-1}\ge 1$, we have
$$ a_0^{t_1-t_0}a_1^{t_2-t_1}\cdots a_{n-1}^{t_{n}-t_{n-1}}\le (a_0a_1\cdots a_{n-1})^{\delta}.$$
\end{lemma}

\begin{lemma}[{see \cite[Lemma 2.5.2]{V87} and also \cite[Theorem 2.1(h)]{Sil87}}]\label{3.3} 
Let $Y$ be a closed subscheme of $V$, and let $\tilde{V}$ be the blowing-up of $V$ along $Y$ with exceptional divisor $E$. Then $\lambda_{Y, v}(\pi(P))=\lambda_{E, v}(P)+O_{v}(1)$ for $P \in \tilde{V} \backslash {\rm Supp} E$. 
\end{lemma}

\begin{lemma}[{see \cite[Lemma 5.4 .24]{Laz17}}]\label{3.4}
Let $X$ be projective variety, $\mathcal{I}$ be a coherent ideal sheaf. Let $\pi: \tilde{X} \rightarrow X$ be the blowing-up of $\mathcal{I}$ with exceptional divisor $E$. Then there exists an integer $p_{0}=p_{0}(\mathcal{I})$ with the property that if $p \geq p_{0}$, then $\pi_{*} \mathcal{O}_{\tilde{X}}(-p E)=\mathcal{I}^{p}$, and moreover, for any divisor $D$ on $X$,
$$H^{i}\left(X, \mathcal{I}^{p}(D)\right)=H^{i}\left(\tilde{X}, \mathcal{O}_{\tilde{X}}\left(\pi^{*} D-p E\right)\right)$$
for all $i \geq 0$.
\end{lemma}

We now prove Theorem \ref{1.8}. The proof basically follows Heier-Levin's proof (see \cite[Page 7]{HL}).

\begin{proof}[Proof of Theorem \ref{1.8}]
Denote by $\mathcal{I}_{i, v}$ the ideal sheaf of $Y_{i, v}$. Let $\pi_{i, v}: \tilde{X}_{i, v} \rightarrow X$ be the blowing-up of $X$ along $Y_{i, v}$ and $E_{i, v}$ be the exceptional divisor on $\tilde{X}_{i, v}$. Fix the real number $\epsilon>0$. Choose a rational number $\delta>0$ such that
$$\delta_X(n+1+\delta)(1+\delta)<\epsilon.$$
Then for a small enough positive rational number $\delta'$ depending on $\delta, \delta \pi^{*} A-\delta' E_{i, v}$ is an $\mathbb{Q}$-ample on $\tilde{X}_{i, v}$ for all $i \in\{0, \ldots, l\}$ and $v \in S$. By the definition of Seshadri constant, there exists a rational number $\epsilon_{i, v}>0$ such that
$$\epsilon_{i, v}+\delta' \geq \epsilon_{Y_{i, v}}(A)$$
and $\pi_{i, v}^{*} A-\epsilon_{i, v} E_{i, v}$ is $\mathbb{Q}$-nef on $\tilde{X}_{i, v}$ for all $0 \leq i \leq l, v \in S$. With such choices, we have $(1+\delta) \pi_{i, v}^{*} A-\left(\epsilon_{i, v}+\delta'\right) E_{i, v}$ is ample $\mathbb{Q}$-divisor on $\tilde{X}_{i, v}$ for all $i, v$. Let $N$ be an integer large enough such that $N(1+\delta) \pi_{i, v}^{*} A$ and $N\left[(1+\delta) \pi_{i, v}^{*} A-\left(\epsilon_{i, v}+\delta'\right) E_{i, v}\right]$ are very ample integral divisors on $\tilde{X}_{i, v}$ for all $0 \leq i \leq l$ and $v \in S$.

Fix $v \in S$, we claim that there are divisors $F_{0, v}, \ldots, F_{n, v}$ on $X$ such that
\begin{itemize}
\item[(a)] $N(1+\epsilon) A \sim F_{i, v}$ and $\pi_{i, v}^{*} F_{i, v} \geq N\left(\epsilon_{i, v}+\delta'\right) E_{i, v}$ on $\tilde{X}_{i, v}$ for all $i=0, \ldots, q$.
\item[(b)] The distributive constant of the family of divisors $\{F_{0, v}, \ldots, F_{q, v}\}$ with respect to $X$ does not exceed $\delta_X$.
\end{itemize}
We construct these divisors by induction as follows. Assume $F_{0, v}, \ldots, F_{j-1, v}$ have been constructed so that the assertion (a) holds for all $i=1,\ldots,j-1$ and the distributive constant of the family $\{F_{0, v}, \ldots, F_{j-1, v}, Y_{j, v}, \ldots, Y_{l, v}\}$ with respect to $X$ does not exceed $\delta_X$. To find $F_{j, v}$, we let $\tilde{F}_{i, v}^{(j)}=\pi_{j, v}^{*} F_{i, v}, i=0, \ldots, j-1$, and $\tilde{Y}_{i, v}^{(j)}=\pi_{j, v}^{*} Y_{i, v}$ for $i=j+1, \ldots, q$. 

We choose $s_{j, v} \in$ $H^{0}\left(\tilde X_{j,v}, N(1+\delta)\pi^*A-N\left(\epsilon_{j, v}+\delta'\right)E_{j, v}\right)$ such that $s_{j, v}$ does not vanish indentically on any irreducible components of any $\left(\bigcap_{i \in I}\tilde F_{i,v}\cap\bigcap_{i \in J}\tilde Y_{i, v}\right)$, where $I\subset\{1,\ldots,j-1\}$ and $J\subset\{j,\ldots,l\}$ are not both empty. 
Put $\tilde F_{j, v}:={\rm div}\left(s_{i, v}\right)+N\left(\epsilon_{j, v}+\delta'\right)\tilde Y_{j, v}$. By Lemma \ref{3.4}, we have, for $N$ big enough,
$$H^{0}\left(X, \mathcal{O}_{X}(N(1+\delta) A) \otimes \mathcal{I}_{j, v}^{N\left(\epsilon_{j, v}+\delta'\right)}\right)=H^{0}\left(\tilde{X}_{j, v}, \mathcal{O}_{\tilde{X}_{j, v}}\left(N\left((1+\delta) \pi_{j, v}^{*} A-\left(\epsilon_{j, v}+\delta'\right) E_{j, v}\right)\right)\right).$$
Then, there exists a divisor $F_{j,v}$ on $X$ such that $\tilde F_{j, v}=\pi^*F_{j,v}$. 

For $I\subset\{1,\ldots,j-1\}$ and $J\subset\{j,\ldots,l\}$, not both empty, we consider an irreducible component $\Gamma$ with maximal dimension of $F_{j,v}\cap\bigcap_{i \in I}F_{i,v}\cap\bigcap_{i \in J} Y_{i, v}$. If $\Gamma\subset Y_{j,v}$ then we have
$$\dim\Gamma\le \dim Y_{j,v}\cap\bigcap_{i \in I}F_{i,v}\cap\bigcap_{i \in J} Y_{i, v}.$$
If $\Gamma\not\subset Y_{j,v}$ then $\pi^*\Gamma\not\subset{\rm div}\left(s_{j, v}\right)$, and hence
$$\dim\Gamma=\dim (\pi^*\Gamma\setminus E_{j,v})\le \dim(\bigcap_{i \in I}\tilde F_{i,v}\cap\bigcap_{i \in J}\tilde Y_{i, v}\setminus E_{j,v})-1\le \dim(\bigcap_{i \in I}F_{i,v}\cap\bigcap_{i \in J}Y_{i, v}\setminus Y_{j,v})-1,$$
(since $\pi$ is isomorphic outside $E_{j,v}$). Therefore, we have
\begin{align*}
\max&\left\{1,\frac{1+\sharp I+\sharp J}{n-\dim\Gamma}\right\}\\
&\le\max\left\{1; \frac{1+\sharp I+\sharp J}{n-\dim Y_{j,v}\cap\bigcap_{i \in I}F_{i,v}\cap\bigcap_{j \in J} Y_{j, v}};\frac{\sharp I+\sharp J}{n-\dim(\bigcap_{i \in I}F_{i,v}\cap\bigcap_{j \in J}Y_{j, v}\setminus Y_{j,v})}\right\}.
\end{align*}
Hence, the distributive constant of $\{F_{0, v}, \ldots, F_{j, v}, Y_{j+1, v}, \ldots, Y_{q, v}\}$ with respect to $X$ does not exceed the distributive of $\{F_{0, v}, \ldots, F_{j-1, v}, Y_{j, v}, \ldots, Y_{q, v}\}$, in particular not exceed $\delta_X$. Also, by the construction, we have $\pi_{j, v}^{*} F_{j, v} \geq N\left(\epsilon_{j, v}+\delta'\right) E_{j, v}$ on $\tilde{X}_{j, v}$. So the claim holds by induction.

Since $\pi_{j, v}^{*} F_{j, v} \geq N\left(\epsilon_{j, v}+\delta'\right) E_{j, v}$ on $\tilde{X}_{i, v}$, by applying Lemma \ref{3.3}, for every point $P \in \tilde{X}_{j, v} \backslash {\rm Supp} E_{j, v}$, we have
\begin{align*}
N \epsilon_{Y_{j, v}}(A) \lambda_{Y_{j, v}, v}\left(\pi_{j, v}(P)\right) & \leq N\left(\epsilon_{j, v}+\delta'\right) \lambda_{Y_{j, v}, v}\left(\pi_{j, v}(P)\right) \\
&=N\left(\epsilon_{j, v}+\delta'\right) \lambda_{E_{j, v}, v}(P) \\
& \leq \lambda_{\pi_{j, v}^{*} F_{j, v}, v}(P)=\lambda_{F_{j, v}, v}\left(\pi_{j, v}(P)\right)
\end{align*}
Then, for every $x\not\in\bigcup_{j,v}Y_{j,v}$, we have
$$ N \sum_{v \in S} \sum_{j=0}^{q} \epsilon_{Y_{j, v}}(A) \lambda_{Y_{j, v}, v}(x)\leq \sum_{v \in S} \sum_{j=0}^{q} \lambda_{F_{j, v}}(x).$$

Note that $\delta_X(n+1+\delta)(1+\delta)<\epsilon$. Therefore, in oder to finish the proof of the theorem, it suffices for us to prove the following claim.
\begin{claim} There is a a proper Zariski-closed subset $Z$ of $X$ such that for all ${\bf x}\in X\setminus Z$,
$$\sum_{v \in S} \sum_{j=0}^{q} \lambda_{F_{i, v}, v}({\bf x}) \leq \delta_X[(n+1)+\delta] h_{N(1+\delta) A}({\bf x}).$$
\end{claim}
By adding more divisors if necessary, we assume that $\bigcup_{i=0}^{q}F_{i, v}=\emptyset$. There is a permutation $I^v=(I^v(0),\ldots,I^v_q)$ of $\{0,\ldots,q\}$ such that 
$$ \lambda_{F_{I^v(0), v}, v}({\bf x})\ge\cdots\ge \lambda_{F_{I^v(0), v}, v}({\bf x}),\ \forall v\in S.$$
Let $l_v$ be the smallest index such that $\bigcup_{i=0}^{l_v}F_{I^v(i), v}=\emptyset$ for each $v\in S$. Then there are indexes $0=t_{0,v}<t_{1,v}<\cdots<t_{n,v}=l_v$ such that
$$\dim\left(\bigcap_{j=0}^{s} F_{I^v(j),v}\right) \cap X=n-u\ \forall\ 1 \leq u \leq n; t_{u-1,v}\le s\le t_{u,v}-1.$$
Then, $\delta_X\ge \frac{t_{u,v}}{u}$ for all $u=1,\ldots, n$.

Denote by $\phi: X \rightarrow \mathbb{P}^{\tilde N_v}(k)$ the canonical embedding associated to the very ample divisor $N(1+\delta) A$ and let $H_{0,v}, \ldots, H_{q,v}$ be the hyperplanes in $\mathbb{P}^{\tilde {N}}(k)$ with $F_{j,v}=\phi^{*} H_{j,v}$ for $j=0, \ldots, q$. We denote $\tilde H_{0,v}, \ldots, \tilde H_{q,v}$ the linear forms defining $H_{0,v}, \ldots, H_{q,v}$ respectively. By Lemma \ref{3.1}, there exist hyperplanes $L_{0,v}, \ldots, L_{n,v}$ with defining linear forms $\tilde{L}_{0,v}, \ldots, \tilde{L}_{n,v}$, such that $\tilde{L}_{0,v}=\tilde H_{I^v(0),v}$, and for every $s \in\{1, \ldots, n\}$, $\tilde{L}_{s,v} \in {\rm span}_{k}\left(\tilde H_{I^v(0),v}, \ldots, \tilde H_{I^v(t_{s}),v}\right)$ and $\phi^{*}L_{0,v}, \ldots, \phi^{*}L_{n,v}$ are located in general position on $X$. Applying Theorem B to $\phi^{*}L_{0,v}, \ldots, \phi^{*}L_{n,v}$, we conclude that, for $\delta >0$, there exists a Zariski closed set $Z$ such that for all ${\bf x} \in X(k) \setminus Z$,
$$\sum_{v \in S} \sum_{i=0}^{n} \lambda_{\phi^{*}L_{i, v}, v}({\bf x}) \leq \delta_X[(n+1)+\delta] h_{N(1+\delta) A}({\bf x}).$$
Hence, for $P=\phi ({\bf x})$, we have
\begin{align*}
\sum_{v \in S}\sum_{i=0}^{q} \lambda_{H_{j,v}}(P)&=\sum_{v \in S}\sum_{j=0}^{l_v} \lambda_{H_{I^v(t_j),v}}(P)+O(1)\\ 
&\le\sum_{v \in S}\sum_{s=0}^{n-1} \left(t_{s+1,v}-t_{s,v}\right) \lambda_{H_{I^v(t_s),v}}(P)+\lambda_{H_{I^v(t_n),v}}(P)+O(1) \\
\text{[by Lemma \ref{3.2}] }&\le\sum_{v \in S}\delta_X\sum_{s=0}^{n-1} \lambda_{H_{I^v(t_s),v}}(P)+\lambda_{H_{I^v(t_n),v}}(P)+O(1) \\
\text{[since $\lambda_{H_{I^v(t_s),v}}(P)\le \lambda_{L_{s,v}}(P)$] }&\le\delta_X\sum_{v \in S}\sum_{s=0}^{n} \lambda_{L_{s,v}}(P)+O(1)\\
&\le\delta_X\sum_{v \in S}\sum_{s=0}^{n} \lambda_{\phi^*L_{s,v}}({\bf x})+O(1)\\
&\le\delta_X[(n+1)+\delta] h_{N(1+\delta) A}({\bf x})+O(1).
\end{align*}
Here, we note that there are only many finite points $x\in X$ such that $h_{N(1+\delta) A}(x)$ is bounded above by a bounded term $O(1)$ and there are only many finite hyperplanes $\tilde H_{s,v}$ ocurring above (althrough the family $\{\tilde H_{s,v}\}$ depends on $x$). Therefore,
 $$ \sum_{v \in S}\sum_{i=0}^{q} \lambda_{F_{j,v}}({\bf x})\le\delta_X[(n+1)+\delta] h_{N(1+\delta) A}({\bf x})$$
for all $x\in X(k)$ outside a Zariski closed proper subset. Then, the claim is proved and the proof of the theorem is follows. 
\end{proof}

\begin{remark}{\rm As we known that there is a close relationship between Diophantine approximation and Nevanlinna theory due to the works of Osgood (see \cite{Os81,Os85}) and Vojta (see \cite{V87,V89,V97}). Especially, Vojta has given a dictionary which provides the correspondences for the
fundamental concepts of these two theories. Via this dictionary, the subspace theorem corresponds to the second main theorem in Nevanlinna theory. By the standard notions in Nevanlinna theory, with the usal arguments, the proof of Theorem \ref{1.8} can be adapted to prove the following generalization of the Second Main Theorem in the side of Nevanlinna theory.}
\end{remark}
\begin{theorem}
Let $X$ be a complex projective variety of $n$-dimension, $Y_{0}, \ldots, Y_{q}$ closed subschemes of $X$. Let $f:\C \rightarrow X$ be a holomorphic map with Zariski dense image, $A$ an ample Cartier divisor on $X$, $\delta$ a positve number, and $\epsilon>0$. Then
$$\int_{0}^{2 \pi} \max _{J} \sum_{j \in J} \epsilon_{Y_j}(A) \lambda_{Y_{j}}(f(r e^{i \theta})) \frac{d \theta}{2 \pi} \leq_{\rm exc}(\delta(n+1)+\epsilon)T_{f, A}(r),$$
where the maximum is taken over all subsets $J$ of $\{0, \ldots, q\}$ such that the distributive constant of the family $\{Y_{j}, j \in J\}$ with respect to $X$ at most $\delta$, and the notation $\leq_{\mathrm{exc}}$ means that the inequality holds for all $r \in(0, \infty)$ outside of a set of finite Lebesgue measure.
\end{theorem}
Here, $T_{f, A}(r)$ is the characteristic function of the holomorphic curve $f$ with respect to the ample line bundle generated by the divisor $A$. This theorem also gives a generalization for the recent second main theorem of  Wang-Cao-Cao \cite[Theorem 1.7]{WCC}.

\noindent{\bf Acknowledgements.} This work was done during a stay of the author at the Vietnam Institute for Advanced Study in Mathematics (VIASM). He would like to thank the staff there, as well as the partially support of VIASM. This research is supported by Vietnam National Foundation for Science and Technology Development under grant number 101.02-2021.12.

\noindent
{\bf Disclosure statement.} No potential conflict of interest was reported by the authors.

\vskip0.2cm
{\footnotesize 
\noindent
{\sc Si Duc Quang}
\vskip0.05cm
\noindent
$^1$Department of Mathematics, Hanoi National University of Education,\\
136-Xuan Thuy, Cau Giay, Hanoi, Vietnam.
\vskip0.05cm
\noindent
$^2$Thang Long Institute of Mathematics and Applied Sciences,\\
Nghiem Xuan Yem, Hoang Mai, HaNoi, Vietnam.
\vskip0.05cm
\noindent
\textit{E-mail}: quangsd@hnue.edu.vn

\end{document}